\documentclass[11pt]{amsart}
 \UseRawInputEncoding 
\usepackage{mathrsfs}
\usepackage{amsfonts}
\usepackage{amsmath}
\usepackage{amssymb}
\usepackage{amsthm}
\usepackage{enumerate}
\usepackage{color}
\usepackage{geometry}
\usepackage{hyperref}
\usepackage{float}
\usepackage{graphicx}
\usepackage{multirow}
\usepackage[numbers,sort&compress]{natbib}
\usepackage{color}

\allowdisplaybreaks
\renewcommand\aa{a}
\newcommand\bb{b}

\newcommand\cc{c}
\newcommand{\chart}{\mathsf{char}}
\newcommand\dd{d}
\newcommand\hh{h}

\newcommand\HS[1]{\leavevmode\null\hspace{#1mm}}
\newcommand\Id[1]{\mathsf{Id}(#1)}

\newcommand\ii{i}

\newcounter{ITEM}

\newcommand\jj{j}

\newcommand\mm{m}
\newcommand\mA{\mathcal{A}}

\newcommand\mN{\mathcal{N}}

\newcommand\nn{n}

\newcommand\pp{p}
\newcommand\qq{q}

 \newcommand\Span{\mathsf{span}}

\newcommand\wdots{, ...\HS{0.2}, }
\newcommand\xx{x}

\newcommand\yy{y}
\newcommand\zz{z}

\def\WEN{\color{black}}

\usepackage{graphicx}
\usepackage{xy} \xyoption{all}
\usepackage{xcolor}

\newtheorem{thm}{Theorem}[section]
\newtheorem{cor}[thm]{Corollary}

\newtheorem{lem}[thm]{Lemma}

\theoremstyle{definition}
\newtheorem{defn}[thm]{Definition}

\newtheorem{exas}[thm]{Example}

\let\phi\varphi
\begin{document}
\title{LIE NILPOTENT NOVIKOV ALGEBRAS AND LIE SOLVABLE
LEAVITT PATH ALGEBRAS}
\maketitle
\begin{center}
Zerui Zhang \footnote{School of Mathematical Sciences, South China Normal University, Guangzhou 510631, P. R. China. E-mail address: \texttt{zeruizhang@scnu.edu.cn}} and
Tran Giang Nam\footnote{Institute of Mathematics, VAST, 18 Hoang Quoc Viet, Cau Giay, Hanoi, Vietnam. E-mail address: \texttt{tgnam@math.ac.vn}}
\end{center}

\begin{abstract} In this paper, we first  study properties of the lower central chains for Novikov algebras. Then we show that for every Lie nilpotent Novikov algebra~$\mathcal{N}$, the ideal of~$\mathcal{N}$ generated by the set~$\{ab - ba\mid a, b\in \mathcal{N}\}$ is nilpotent. We secondly provide necessary and sufficient conditions on the graph  $E$ and the field  $K$ for which the Leavitt path algebra $L_K(E)$ is Lie solvable. Consequently, we obtain a complete description of Lie nilpotent Leavitt path algebras, and show that the Lie solvability of~$L_K(E)$ and the Lie nilpotency of $[L_K(E),L_K(E)]$ are the same.
\medskip


\textbf{Mathematics Subject Classifications}: 16S88, 17A01, 17A30, 17B30

\textbf{Key words}: Novikov algebra; Leavitt path algebra; Lie nilpotent and solvable algebra.
\end{abstract}

\section{Introduction}

It is well known that for an arbitrary algebra~$\mA$, one can always define a multiplication by $[\xx, \yy] = \xx\yy -\yy\xx$, where the juxtaposition denotes multiplication in~$\mA$. And~$\mA$  is called a \emph{Lie-admissible} algebra~\cite{Albert} if~$(\mA, [-,-])$  is a Lie algebra. In this case, we call $(\mA, [-,-])$ the {\it associated Lie algebra}   of  $\mA$.  As is well known, Novikov and associative algebras are Lie-admissible algebras.  A natural and interesting problem is to determine the structure of a Lie-admissible algebra when its associated Lie algebra has some properties.

Although, in general,  answering this question seems to be a quite difficult task, there have been obtained a number of interesting results regarding the question among which we mention, for example, the following ones. In \cite{jen1}, with suitable definitions of ``commutator ideal", many of the properties of commutator subgroups had analogues in the theory of associative algebras. In particular,  Jennings~\cite{jen1} was interested in extending the notions of ``nilpotent group" and ``solvable group" to rings, and proved that: (1)  if $\mA$ is an associative algebra, not of characteristic $2$ (the latter condition is shown to be indispensable), then $\mA$ is solvable if  its associated Lie algebra is;  (2) if $\mA$ is an associative algebra whose associated Lie algebra is nilpotent, then
the ideal~$\mA\circ \mA$  of~$\mA$ generated by the set~$\{ab-ba\mid \aa,\bb\in\mA\}$  is nilpotent. Sharma and Srivastava~\cite{lie-solvable} proved that if $\mA$ is an associative algebra over a field $K$ whose associated Lie algebra is
solvable, and if the characteristic of $K$ is neither~$2$ nor~$3$, then~$\mA\circ \mA$ is nil. Riley~\cite{Riley} proved that for an associative algebra~$\mA$ over a field of characteristic $p >0$,  the ideal~$\mA\circ \mA$ is nil of bounded index if the associated Lie algebra of $\mA$ is either nilpotent, or  solvable with $p> 2$.

Novikov algebras arise in many areas of mathematics and physics. We refer to \cite{Burde2006}, and the references therein, for more information on these topics. Bai and Meng
classified the Novikov algebras over $\mathbb{C}$ of dimension up to $3$ in \cite{baimeng} (we should note that Novikov algebras are Lie-admissible, hence can be classified according to their associated Lie algebras), and   classified  the  complete Novikov algebras over $\mathbb{C}$ that have a nilpotent associated Lie algebra in \cite{bai}. Burde and Graaf~\cite{Burde} classified the $4$-dimensional Novikov algebras over~$\mathbb{C}$ that have a nilpotent associated Lie algebra (not just the complete ones).

A Leavitt path algebra is a universal algebra constructed from a directed graph, simultaneously introduced by Abrams and Aranda Pino in \cite{ap:tlpaoag05} as well as by Ara, Moreno, and Pardo  in \cite{amp:nktfga}. The Leavitt path algebras generalize the Leavitt algebras $L_K(1, n)$ of \cite{leav:tmtoar}, and also contain many other interesting classes of algebras.  In addition, Leavitt path algebras are intimately related to graph $C^*$-algebras (see \cite{r:ga}). During  the past fifteen years, Leavitt path algebras have become a topic of intense investigation by mathematicians from across the mathematical spectrum.  We refer the reader to \cite{a:lpatfd} and \cite{AAS} for a detailed history and overview of Leavitt path algebras. Recently there have been obtained a number of interesting results regarding the associated Lie algebra of a Leavitt path algebra. Abrams and Funk-Neubauer~\cite{leav-sim-lie} gave a criterion on the simplicity of the Lie algebra~$[L_K(1, n), L_K(1, n)]:=\Span_{K}\{\aa\bb-\bb\aa \mid \aa,\bb\in L_K(1, n)\}$.  Abrams and Mesyan~\cite{leav-lie} extended the result to all simple Leavitt path algebras  $L_K(E)$ of row-finite graphs $E$ over a field $K$ by giving necessary and sufficient conditions on~$E$ and~$K$ which determine the simplicity of the Lie algebra~$[L_{K}(E), L_{K}(E)]$. In \cite{Nam} the second author generalized Abrams and Mesyan's result to simple Steinberg algebras.

The current paper is a continuation of the investigation of the structure of Novikov and Leavitt path algebras when their associated Lie algebras are nilpotent or solvable. More  precisely,  based on Jennings's idea~\cite{jen1} for associative algebras, we investigate the nilpotency of the associated Lie algebra of a Novikov algebra via central chains of ideals, and obtain an analogue of Jennings's result \cite[Theorems 3.9 and 6.6]{jen1} that the associated Lie algebra of a Novikov algebra~$\mN$ is nilpotent if and only if~$\mN$ is of finite class. Consequently, for  every  Lie nilpotent Novikov algebra $\mN$, the ideal of~$\mN$ generated by~$\{ab-ba\mid \aa,\bb\in\mN\}$ is nilpotent (Theorem \ref{nil-finclass}). We should mention that Jennings's proof was based essentially on the associativity. However, Novikov algebras are in general not associative algebras, and hence some different, novel techniques have to be applied in places. We also find a property of Lie nilpotent Novikov algebras on the products of commutator ideals that does not hold in general for Lie nilpotent associative algebras (Theorem \ref{prod-com-id}). For Leavitt path algebra case, we provide necessary and sufficient conditions on $E$ and $K$ for which the associated Lie algebra of the Leavitt path algebra $L_K(E)$ is solvable (Theorem \ref{sol-LPAs} and and Corollary \ref{sol-LPAs1}). Consequently,  we obtain a complete description of Lie nilpotent Leavitt path algebras (Corollary \ref{nil-LPAs}), and
show that the Lie solvability of~$L_K(E)$ and the Lie nilpotency of $[L_K(E),L_K(E)]$ are the same (Corollary \ref{nil-LPAs1}).

\section{Lie nilpotent Novikov algebras}
Recall (see, e.g., \cite{Gelfand,Balinski-Novikov}) that an algebra~$\mN$ is called a (left) \emph{Novikov algebra} if~$\mN$ satisfies the following identities
$$
\xx  (\yy  \zz)-(\xx  \yy)  \zz
=\yy  (\xx  \zz)- (\yy  \xx)  \zz \ \mbox{(left symmetry)},
$$
$$
(\xx  \yy)  \zz=(\xx  \zz)  \yy \ \mbox{(right commutativity)}
$$
for all~$\xx$, $\yy$ and~$\zz$ $\in \mN$.  A Novikov algebra is not an associative algebra  in general, but  it is a Lie-admissible algebra. However, its left symmetry is, to some extent, a weakened version of the associativity, and the right commutativity might bring some interesting properties that an associative algebra does not have in general.  So it is a natural problem to consider what kind of results for associative algebras also holds for Novikov algebras. There have been obtained a number of
interesting results regarding this question among which we mention, for example, the following ones. Kegel proved in~\cite{ass} that an associative ring represented as a sum of two nilpotent subrings is nilpotent; Pchelintsev generalized  Kegel's theorem for alternative algebras, substituting nilpotency by solvability~\cite{Pch-alter};  Shestakov and the first author~\cite{nov-sol} proved that a (left) Novikov algebra represented as a sum of two right nilpotent (equivalently,   solvable) subalgebras is right nilpotent (equivalently,   solvable).  In this section, we shall see that Novikov algebras share some common properties as associative algebras when it comes to the certain properties of lower central chains.

\subsection{Central chains of ideals in a Novikov algebra}\label{subsec-central-ideal}
The aim of this subsection is to show that some properties of lower central chain of associative algebras also hold for Novikov algebras. These properties will be very useful for the study of Lie nilpotent Novikov algebras in the next subsection.

Let $\mA$ be an arbitrary Lie-admissible algebra over a given field $K$. We define
$$[\aa,\bb]=\aa\bb-\bb\aa$$ for all $a$ and $b\in\mA$.
For all subspaces $A$ and $B$ of~$\mA$, we define
$$[A,B]=\Span\{[\aa,\bb] \mid \aa\in A, \bb\in B\}  \  \mbox{ and }  \  AB=\Span\{\aa\bb\mid \aa\in A, \bb\in B\}.$$
We call a space~$V\subseteq \mA$ a \emph{Lie ideal} of~$\mA$  if we have~$[V,\mA]\subseteq V$.
Finally, for all subspaces  $A$ and $B$ of~$\mA$, we define
$$A\circ B=\Id{[A,B]},$$
that is, the ideal of~$\mA$ generated by~$[A,B]$. Following the idea of Jennings \cite{jen1}, we call~$A\circ B$ the  \emph{commutator ideal} of~$A$ and~$B$.  We clearly have~$A\circ B=B\circ A$.

The following lemma provides us with a description of commutator ideals of Novikov algebras. 

\begin{lem}\label{com-id}
Let~$\mN$ be a Novikov algebra, and let~$A$ and $B$ be two Lie ideals of~$\mN$. Then we have~$A\circ B=[A,B] +\mN[A,B]$.
\end{lem}
\begin{proof}
Obviously, it suffices to show that~$[A,B] +\mN[A,B]$ is an ideal of~$\mN$. Since $\mN$ is Lie-admissible,   we have
$$[[\xx,\yy],\zz]+[[\yy,\zz],\xx]+[[\zz,\xx],\yy]=0$$  for all~$\xx,\yy,\zz\in\mN$.	In particular, for all~$\aa\in A$,  $\bb\in B$ and~$\xx\in \mN$, we have
$$[\aa,\bb]\xx-\xx[\aa,\bb]
=[[\aa,\bb],\xx]=[[\aa,\xx],\bb]+[\aa,[\bb,\xx]].$$
Since~$[\aa,\xx]\in A$ and~$[\bb,\xx]\in B$, we obtain  that
$$[\aa,\bb]\xx=\xx[\aa,\bb]+[[\aa,\xx],\bb]+[\aa,[\bb,\xx]]\in [A,B] +\mN[A,B].$$
For all~$\yy\in \mN$, by the right commutativity, we obtain that
$$(\xx[\aa,\bb])\yy=(\xx\yy)[\aa,\bb]\in [A,B] +\mN[A,B], $$
and so, by the left symmetry and by the above reasoning, we deduce
\begin{align*}
\yy(\xx[\aa,\bb])
=&\yy(-[[\aa,\xx],\bb]-[\aa,[\bb,\xx]]+[\aa,\bb]\xx)&\\
=&-\yy[[\aa,\xx],\bb]-\yy[\aa,[\bb,\xx]]+\yy([\aa,\bb]\xx)&\\
=&-\yy[[\aa,\xx],\bb]-\yy[\aa,[\bb,\xx]]
+(\yy[\aa,\bb])\xx+[\aa,\bb](\yy\xx)-([\aa,\bb]\yy)\xx&\\
\in &\, [A,B] +\mN[A,B].&
\end{align*}
This implies that $[A,B] +\mN[A,B]$ is an ideal, and so we have~$A\circ B=[A,B] +\mN[A,B]$, thus finishing the proof.
\end{proof}

Equipped with the notion of commutator ideals, we are now able to introduce the notion of central chains of ideals of a Lie-admissible algebra~$\mA$.

Let
\begin{equation}\label{def-cent}
\mA=A_1\supseteq A_2 \supseteq \cdots \supseteq A_\mm
\supseteq A_{\mm+1}=(0)
\end{equation}
be a chain of ideals of~$\mA$. Such a chain is called a \emph{central chain of ideals} if we have
\begin{equation}\label{cent-cond}
\mA\circ A_\ii\subseteq A_{\ii+1} \ \ \   (\ii=1,2\wdots m) .
\end{equation}
We shall soon see that Novikov algebras which possess central chains  of ideals have special properties; we investigate some of them by considering a particular central chain.
\begin{defn}\label{defi-fc}
For every Lie-admissible algebra~$\mA$ we form a series of ideals
\begin{equation}\label{def-hi}
H_1:=\mA, \  H_{\ii+1}:=H_\ii\circ \mA \mbox{ for } \ii\geq 1.
\end{equation}
We say that~$\mA$ is of \emph{finite class} if~$H_n=(0)$ for some positive integer~$n$. For the minimal integer~$n$ such that~$H_n=(0)$, we call~$n-1$ the \emph{class} of~$\mA$, and call
\begin{equation}\label{lower-cent}
\mA=H_1\supseteq H_2\supseteq \cdots \supseteq H_{n-1}\supseteq H_n=(0)
\end{equation}
the \emph{lower central chain} of~$\mA$.  To avoid too many repetitions, we shall fix the notation of~$H_\ii$ for all~$\ii\geq 1$.
\end{defn}

It is well known (see, e.g., \cite[Theorem 2.1]{jen1}) that  an associative algebra is of finite class if and only if it has a central chain of ideals, the proof of which is quite easy and holds for every Lie-admissible algebra.  In light of this note, we shall focus on the study of properties of the lower central chain of   Novikov algebras that are of finite class.

For associative algebras that are of finite class, Jennings \cite{jen1} showed nice connections between the central  chains~\eqref{def-cent} and~\eqref{lower-cent}. In particular, in \cite[Theorems 3.3 and 3.4]{jen1} Jennings proved that $H_p A_q\subseteq A_{\pp+\qq-1}$ and $H_p\circ A_q\subseteq A_{\pp+\qq}$ for all integers~$\pp,\qq\geq1$, and the proof was based essentially on the associativity. In the following theorem, we provide an analogue of this result for  Novikov algebras with new techniques.
\begin{lem}\label{circ-pro}
Let~$\mN=A_1\supseteq A_2 \supseteq \cdots \supseteq A_\mm
\supseteq \cdots $ be a chain of ideals of a Novikov algebra~$\mN$  satisfying~$\mN \circ A_\ii\subseteq A_{\ii+1}$ for all~$\ii\geq 1$.
Define a series of ideals inductively by the rule
$H_1=\mN$, $H_{\ii+1}=\mN \circ H_\ii \mbox{ for } \ii\geq 1$.
Then  we have~$H_\pp\circ A_q\subseteq A_{\pp+\qq}$. In particular, we have~$H_\pp\circ H_\qq\subseteq H_{\pp+\qq}$.
\end{lem}
\begin{proof}
We use induction on~$\pp$ to prove the claim. For~$\pp=1$,  the claim immediately follows from the definition of central chain of ideals. Now we proceed inductively. Assume that~$\pp>1$.  Obviously, for all~$\pp,\qq\geq1$, all~$H_\pp$ and~$A_\qq$ are Lie ideals of~$\mN$. So,  by Lemma~\ref{com-id}, we have
$$H_\pp \circ A_\qq=[H_\pp,A_\qq] +\mN[H_\pp,A_\qq].$$
Since~$A_{\pp+\qq}$ is an ideal of~$\mN$, it suffices to show~$[\hh, \aa]\in A_{\pp+\qq}$  for all~$\hh\in H_\pp$, $\aa\in A_\qq$.
Consider the following cases.
Case 1: If~$\hh=[\hh_{\pp-1},\xx]$ for some elements~$\hh_{\pp-1}\in H_{\pp-1}$ and~$\xx\in \mN$, then by the induction hypothesis, we have
\begin{multline*}
[\hh,\aa]
= [[\hh_{\pp-1},\xx],\aa]
= [[\hh_{\pp-1},\aa],\xx]+ [\hh_{\pp-1},[\xx,\aa]] \\
\in  [A_{\pp+\qq-1}, \mN]+ [H_{\pp-1}, A_{\qq+1}]
\subseteq  A_{\pp+\qq-1}\circ \mN+ H_{\pp-1}\circ A_{\qq+1}
\subseteq  A_{\pp+\qq}.
\end{multline*}
Case 2: If~$\hh=\yy[\hh_{\pp-1},\xx]$ for some elements~$\hh_{\pp-1}\in H_{\pp-1}$ and~$\xx,\yy\in \mN$, then we denote~$[\hh_{\pp-1},\xx]$ by~$\hh'$. Obviously, $\hh'$ is in~$H_\pp$.   By Case 1, for all~$\zz\in A_\qq$, we have~$[\hh',\zz]\in A_{\pp+\qq}$. In particular, we have
$$ [\hh',\aa]\yy+\yy[\hh',\aa] -[\hh',\yy\aa]\in A_{\pp+\qq}.$$
By the induction hypothesis, we obtain that
\begin{align*}
[\hh,\aa]
=&(\yy\hh')\aa-\aa(\yy\hh')&\\
=&(\yy\aa)\hh'-(\aa\yy)\hh'
-\yy(\aa\hh')+(\yy\aa)\hh'&\\
=&(\yy\aa)\hh'-\hh'(\aa\yy)+[\hh',\aa\yy]
-\yy(\hh'\aa)+\yy[\hh',\aa]
+\hh'(\yy\aa)-[\hh',\yy\aa]&\\
=&(\yy\aa)\hh'-\hh'(\aa\yy)
-\yy(\hh'\aa)+(\hh'\yy)\aa+\yy(\hh'\aa)
-(\yy\hh')\aa&\\
&+[\hh',\aa\yy]+\yy[\hh',\aa]-[\hh',\yy\aa]&\\
=&(\hh'\yy)\aa-\hh'(\aa\yy)
+[\hh',\aa\yy]+\yy[\hh',\aa]-[\hh',\yy\aa]\  (\mbox{by right commutativity})&\\
=&(\hh'\aa)\yy-\hh'(\aa\yy)
+[\hh',\aa\yy]+\yy[\hh',\aa]-[\hh',\yy\aa]&\\
=&(\aa\hh')\yy+[\hh',\aa]\yy-\hh'(\aa\yy)
+[\hh',\aa\yy]+\yy[\hh',\aa]-[\hh',\yy\aa]&\\
=&(\aa\yy)\hh'-\hh'(\aa\yy)
+[\hh',\aa]\yy+[\hh',\aa\yy]+\yy[\hh',\aa]
-[\hh',\yy\aa]&\\
=& [\hh',\aa]\yy+\yy[\hh',\aa] -[\hh',\yy\aa]\in A_{\pp+\qq}.&
\end{align*}
\WEN
Therefore, we have~$H_\pp\circ A_\qq\subseteq A_{\pp+\qq}$, thus finishing the proof.
\end{proof}

Before going further, we shall prove two interesting identities that will be useful in the sequel.
\begin{lem}\label{lem-id}
Let~$\mN$ be an arbitrary Novikov algebra and let~$\xx,\yy,\zz$ be elements of~$\mN$.  Then we have~$[\xx,\yy]\zz+[\yy,\zz]\xx+[\zz,\xx]\yy=0$ and~$\xx[\yy,\zz]+\yy[\zz,\xx]+\zz[\xx,\yy]=0$.
\end{lem}
\begin{proof}
  By right commutativity, we have
  \begin{align*}
    [\xx,\yy]\zz+[\yy,\zz]\xx+[\zz,\xx]\yy
    =&(\xx\yy)\zz-(\yy\xx)\zz
    +(\yy\zz)\xx-(\zz\yy)\xx
    +(\zz\xx)\yy-(\xx\zz)\yy&\\
     =&(\xx\yy)\zz-(\xx\zz)\yy
     +(\yy\zz)\xx-(\yy\xx)\zz
    +(\zz\xx)\yy -(\zz\yy)\xx=0.&
  \end{align*}
 Moreover, since every Novikov algebra is Lie-admissible, we have
 $$[\xx,\yy]\zz-\zz[\xx,\yy]+[\yy,\zz]\xx-\xx[\yy,\zz]
 +[\zz,\xx]\yy-\yy[\zz,\xx]=[[\xx,\yy],\zz]+[[\yy,\zz],\xx]+[[\zz,\xx],\yy]=0.$$
 Therefore, we obtain that
 $$\xx[\yy,\zz]+\yy[\zz,\xx]+\zz[\xx,\yy]=[\xx,\yy]\zz+[\yy,\zz]\xx
 +[\zz,\xx]\yy=0,$$
 thus finishing the proof.
\end{proof}

Recall that for an arbitrary subspace~$V$ of a Novikov algebra $\mN$, we define~$V^1=V$ and~$V^{\nn}=\sum_{1\leq\ii\leq\nn-1}V^{\ii} V^{\nn-\ii}$ for all $\nn\geq2$. Then~$V$ is called \emph{nilpotent} if~$V^{\nn}=(0)$ for some positive integer~$\nn$. We are now in position to provide the main result of this subsection, which shows that if~$\mN$ is a Novikov algebra that is of finite class, then the commutator ideal~$\mN\circ \mN$ is nilpotent.

\begin{thm}\label{th-pro}
 With the notation of Lemma~\ref{circ-pro}, we have~$H_\pp H_\qq\subseteq H_{\pp+\qq-1}$ for all positive integers~$\pp$ and~$\qq$. Moreover, if~$\mN$ is of finite class, then $H_2$ is nilpotent of nilpotent index less or equal to the class of~$\mN$.
\end{thm}
\begin{proof}
We first use induction on~$\min\{\pp,\qq\}$ to prove~$H_\pp H_\qq\subseteq H_{\pp+\qq-1}$. If~$\min\{\pp,\qq\}=1$, then it is clear because~$H_\pp$ and~$H_\qq$ are ideals of~$\mN$.
Now we proceed inductively. Assume that~$\pp,\qq\geq 2$. By Lemma~\ref{circ-pro}, for all $\hh_\pp\in H_\pp$ and $\hh_\qq\in H_\qq$, we have
$$\hh_\pp\hh_\qq-\hh_\qq\hh_\pp=[\hh_\pp, \hh_\qq]
\subseteq H_\pp\circ H_\qq\subseteq H_{\pp+\qq}\subseteq H_{\pp+\qq-1}.$$	In other words, $\hh_\pp\hh_\qq$ lies in~$H_{\pp+\qq-1}$ if and only if so does~$\hh_\qq\hh_\pp$. So, without loss of generality, we may assume~$\pp\leq \qq$.
If~$\hh_\pp=\xx_2[\hh_{\pp-1},\xx_1]$ for some~$\xx_1,\xx_2\in\mN$ and~$\hh_{\pp-1}\in H_{\pp-1}$, then we have
$$\hh_\pp\hh_\qq
=(\xx_2[\hh_{\pp-1},\xx_1])\hh_\qq
=(\xx_2\hh_\qq)[\hh_{\pp-1},\xx_1]
=\hh_\qq'[\hh_{\pp-1},\xx_1]$$
for~$\hh_\qq'=\xx_2\hh_\qq\in H_\qq$. So~$\hh_\pp\hh_\qq$ lies in~$H_{\pp+\qq-1}$ if and only if so does~$\hh_\qq'[\hh_{\pp-1},\xx_1]$. By the above reasoning, it suffices to show~$[\hh_{\pp-1},\xx_1]\hh_\qq'\in H_{\pp+\qq-1}$.
So we may assume~$\hh_\pp=[\hh_{\pp-1},\xx]$ for some~$\xx\in\mN$ and~$\hh_{\pp-1}\in H_{\pp-1}$. By  Lemmas~\ref{circ-pro} and~\ref{lem-id} and  the induction hypothesis, we deduce
$$
\hh_\pp\hh_\qq=[\hh_{\pp-1},\xx]\hh_\qq
= -[\hh_\qq,\hh_{\pp-1}]\xx-[\xx,\hh_\qq]\hh_{\pp-1}
\in  H_{\pp-1+\qq}\xx+H_{\qq+1} H_{\pp-1}
\subseteq  H_{\pp+\qq-1}.
$$
For the second claim, we shall use induction on~$\mm$ to show~$(H_2)^m\subseteq H_{m+1}$. For~$\mm=1$, there is nothing to prove. Assume that~$\mm\geq 2$. By the induction hypothesis, we have
\begin{align*}
   (H_2)^\mm&=\sum_{1\leq\ii\leq\mm-1}(H_2)^{\ii} (H_2)^{\mm-\ii}
\subseteq\sum_{1\leq\ii\leq\mm-1}H_{\ii+1} H_{\mm-\ii+1}&\\
&\subseteq\sum_{1\leq\ii\leq\mm-1}H_{\ii+1+\mm-\ii+1-1}
=H_{\mm+1}.&
\end{align*}
Since~$\mN$ is of finite class,  we may assume $H_{\mm+1}=(0)$ for some integer $\mm$, and so $H_2$ is nilpotent of index not greater than~$\mm$, thus finishing the proof.
\end{proof}

\subsection{Lie nilpotent Novikov algebras}\label{subsec-lie-nilpotent-nov} In this subsection, based on Theorem \ref{th-pro}, we show that a Novikov algebra $\mN$ is Lie nilpotent if and only if~$\mN$ is of finite class (Theorem \ref{nil-finclass}), which is an analogue of Jennings's result \cite[Theorem 6.6]{jen1} for associative algebras. We also find a property of  Novikov algebras on the products of commutator ideals that does not hold in general for associative algebras (Theorem \ref{prod-com-id}).

We begin this subsection by recalling useful notions. Let~$\mA$ be a Lie-admissible algebra. We define
$$\mA_{[1]}=\mA \ \mbox{ and }\  \mA_{[\ii+1]}=[\mA,\mA_{[\ii]}] \text{ for all } \ii\geq 1.$$ We call~$\Id{\mA_{[\ii]}}$ the \emph{$i$th commutator ideal} of~$\mA$. And the algebra~$\mA$ is called \emph{Lie nilpotent} if~$\mA_{[\ii]}=(0)$ for some integer~$\ii$.

For every integer~$\ii\geq1$, the linear space~$\mA_{[\ii]}$ is obviously a Lie ideal of~$\mA$. Therefore, by Lemma~\ref{com-id}, we immediately obtain the construction of~$\Id{\mA_{[\ii]}}$ when~$\mA$ is a Novikov algebra.
\begin{lem}\label{lem-mni}
For every Novikov algebra~$\mN$, we have~$\Id{\mN_{[\ii]}}=\mN_{[\ii]}+\mN\mN_{[\ii]}$ for all integer~$\ii\geq 1$.
\end{lem}
Lots of interests have been attracted by the subject of commutator ideals of associative algebras.  Etingof, Kim and Ma~\cite{universal-lie-nilpotent} studied the quotient of a free algebra by its $i$-th commutator ideal, and studied the products of such commutator ideals. Kerchev~\cite{filtration} studied the filtration of a free algebra by its associative lower
central series.  We refer to~\cite{BJ,pro-com} and the references therein for a detailed history and overview of this direction.  It is worth mentioning that there is also a series of interesting results on the solvability or nilpotency of Poisson algebras with respect to their Lie structures~\cite{poi1,poi2,poi3}. In light of these notes, it is natural and interesting to study the commutator ideals of Novikov algebras. Let us begin with the following easy observation.

\begin{lem}\label{lem-mni1}
For every Novikov algebra~$\mN$, we have $\mN_{[2]}\mN_{[\ii]}\subseteq \Id{\mN_{[\ii+1]}}$ for all integer~$\ii\geq 1$.
\end{lem}
\begin{proof}
For all~$\aa,\bb\in\mN$, for all~$\cc_\ii\in\mN_{[\ii]}$, by Lemma~\ref{lem-id}, we have
$$
[\aa,\bb]\cc_\ii=-[\bb,\cc_\ii]\aa-[\cc_\ii,\aa]\bb
 \in \Id{\mN_{[\ii+1]}},$$
 thus finishing the proof.
\end{proof}

Recall that for an arbitrary Novikov algebra, we define~$H_1=\mN$ and define~$H_{\ii+1}=\mN\circ H_\ii$ for all~$\ii\geq 1$. By induction on~$\ii$, it is straightforward to show~$\Id{\mN_{[\ii]}}\subseteq H_\ii$.  In the following theorem we show that~$\Id{\mN_{[\ii]}}\supseteq H_\ii$.  In particular, Theorem~\ref{th-pro} then describes the products of the commutator ideals of~$\mN$, which is in general not true for Lie nilpotent associative algebras.

\begin{thm}\label{prod-com-id}
Let~$\mN$ be a Novikov algebra, and define a series of ideals by the rule
$H_1=\mN$, $H_{\ii+1}=\mN \circ H_\ii \mbox{ for all } \ii\geq 1$. Then we have~$\Id{\mN_{[\ii]}}=H_\ii$ for all integer~$\ii\geq 1$. In particular, we have~$\Id{\mN_{[\pp]}}\Id{\mN_{[\qq]}}\subseteq \Id{\mN_{[\pp+\qq-1]}}$ for all integers~$\pp,\qq\geq 1$.
\end{thm}
\begin{proof}
It suffices to prove that~$H_\ii\subseteq\Id{\mN_{[\ii]}}$ for all integer~$\ii\geq 1$, and we use induction on~$\ii$ to prove this claim. For~$\ii\leq 2$, it is clear. Assume that~$\ii\geq 3$. By the induction hypothesis and  Lemma~\ref{lem-mni}, we have that~$H_\ii$ is generated by
$$[\mN, H_{\ii-1}]=[\mN, \mN_{[\ii-1]}]+[\mN, \mN\mN_{[\ii-1]}]
=\mN_{[\ii]}+[\mN, \mN\mN_{[\ii-1]}].$$
Therefore, it is enough to show~$[\mN, \mN\mN_{[\ii-1]}]\subseteq \Id{\mN_{[\ii]}}$. Indeed, for all~$\aa,\bb,\cc\in \mN$ and for all~$\dd_{\ii-2}\in\mN_{[\ii-2]}$, by Lemma \ref{lem-mni1}, we  deduce
\begin{align*}
[\aa, \bb[\cc,\dd_{\ii-2}]]
=&\aa(\bb[\cc,\dd_{\ii-2}])-(\bb[\cc,\dd_{\ii-2}])\aa&\\
=&\aa([\bb,[\cc,\dd_{\ii-2}]]
+[\cc,\dd_{\ii-2}]\bb)-(\bb\aa)[\cc,\dd_{\ii-2}]&\\
=&\aa[\bb,[\cc,\dd_{\ii-2}]]
+\aa([\cc,\dd_{\ii-2}]\bb)-(\bb\aa)[\cc,\dd_{\ii-2}]&\\
=&\aa[\bb,[\cc,\dd_{\ii-2}]]
+(\aa[\cc,\dd_{\ii-2}])\bb+[\cc,\dd_{\ii-2}](\aa\bb)
-([\cc,\dd_{\ii-2}]\aa)\bb-(\bb\aa)[\cc,\dd_{\ii-2}]&\\
=&\aa[\bb,[\cc,\dd_{\ii-2}]]
+(\aa\bb)[\cc,\dd_{\ii-2}]-(\bb\aa)[\cc,\dd_{\ii-2}]
+[\cc,\dd_{\ii-2}](\aa\bb)-([\cc,\dd_{\ii-2}]\aa)b&\\
=&\aa[\bb,[\cc,\dd_{\ii-2}]]
+[\aa,\bb][\cc,\dd_{\ii-2}]
+[\cc,\dd_{\ii-2}](\aa\bb)-(\aa[\cc,\dd_{\ii-2}])b
+[\aa,[\cc,\dd_{\ii-2}]]b&\\
=&\aa[\bb,[\cc,\dd_{\ii-2}]]
+[\aa,\bb][\cc,\dd_{\ii-2}]
+[\cc,\dd_{\ii-2}](\aa\bb)-(\aa\bb)[\cc,\dd_{\ii-2}]
+[\aa,[\cc,\dd_{\ii-2}]]b&\\
=&\aa[\bb,[\cc,\dd_{\ii-2}]]
+[\aa,\bb][\cc,\dd_{\ii-2}]
-[\aa\bb, [\cc,\dd_{\ii-2}]]
+[\aa,[\cc,\dd_{\ii-2}]]b&\\
\in &\Id{\mN_{[\ii]}} +\mN_{[2]}\mN_{[\ii-1]}\subseteq \Id{\mN_{[\ii]}}.&
\end{align*}
Then, by Theorem~\ref{th-pro}, we immediately obtain the statement, thus finishing the proof.
\end{proof}

Consequently, we receive the main theorem of this section, providing that the notions of ``finite class" and ``Lie nilpotency" for Novikov algebras are the same.

\begin{thm}\label{nil-finclass}
A Novikov algebra~$\mN$ is Lie nilpotent if and only if~$\mN$ is of finite class. Consequently, for every Lie nilpotent Novikov algebra $\mN$, the ideal of~$\mN$ generated by~$\{ab-ba\mid \aa,\bb\in\mN\}$ is nilpotent.
\end{thm}
\begin{proof} By Theorem~\ref{prod-com-id}, we immediately obtain that a Novikov algebra is Lie nilpotent if and only if it is of finite class. Then, by Theorem~\ref{th-pro}, the second claim follows, thus finishing the proof.
\end{proof}

\section{Lie solvable Leavitt path algebras}
The main goal of this section is to provide necessary and sufficient conditions on $E$ and $K$ for which the Leavitt path algebra $L_K(E)$ is Lie solvable (Theorem \ref{sol-LPAs} and Corollary \ref{sol-LPAs1}). Consequently,  we obtain a complete description of Lie nilpotent Leavitt path algebras (Corollary \ref{nil-LPAs}), and
show that the Lie solvability of~$L_K(E)$ and the Lie nilpotency of $[L_K(E),L_K(E)]$ are the same (Corollary \ref{nil-LPAs1}).

We begin this section by recalling some useful notions. Let $K$ be a field and $\mA$ an associative algebra over $K$. Then  $\mA$  becomes a Lie algebra under the operation $[x, y] = xy - yx$ for all $x, y \in \mA$,  and~$(\mA, [-,-])$ is  called the {\it associated Lie algebra} of $\mA$. If $\mA$ is unital, then we denote by $\mathfrak{gl}(n,\mA)$ the associated Lie algebra of~$M_n(\mA)$ with~$\mA$-basis $\{E_{ij}\mid 1\leq \ii,\jj\leq \nn\}$, where~$E_{ij}$ is the matrix that the entry in the $i$th row and $j$th column is 1 while the other entries are 0. The following result is well known; and just for the reader's convenience, we reproduce it here.

\begin{lem}\label{mat-sol}  Let $K$ be an arbitrary field and let $\mA$ be either the field $K$ or the Laurent polynomial algebra~$K[x,x^{-1}]$. Then the following holds:

$(1)$ If~$\chart(K)\neq 2$, then for every integer~$n\geq 2$, the Lie algebra $\mathfrak{gl}(n, \mA)$ is not solvable;

$(2)$ If $\chart(K)=2$, then~$\mathfrak{gl}(2, \mA)$ is solvable but not nilpotent, and for every integer~$n\geq 3$, the Lie algebra $\mathfrak{gl}(n, \mA)$ is not solvable.
\end{lem}
\begin{proof} We first note that
since $[a E_{ij},  b E_{pq}] =ab[E_{ij}, E_{pq}]$ for all $a, b\in \mA$, the Lie algebra $\mathfrak{gl}(n, \mA)$ is solvable (respectively, nilpotent) if and only if~$\mathfrak{gl}(n, K)$ is solvable (respectively, nilpotent).

(1) Assume that $\chart(K)\neq 2$. We then have
$$[E_{12},E_{21}]=E_{11}-E_{22}, \ [E_{11}-E_{22}, E_{12}]=2E_{12},
\ [E_{11}-E_{22}, E_{21}]=-2E_{21},$$ so
$\mA(E_{11}-E_{22})+\mA E_{12}+\mA E_{21}$ is not Lie solvable.  This implies that for every integer $n\geq 2$, the Lie algebra $\mathfrak{gl}(n,\mA)$ is not solvable.

(2)  Assume that $\chart(K)=2$.  We then have
$$[\mathfrak{gl}(2,K), \mathfrak{gl}(2,K)]=\text{span}\{E_{11}+E_{22}, E_{12}, E_{21}\}$$
and $$[[\mathfrak{gl}(2,K), \mathfrak{gl}(2,K)],[\mathfrak{gl}(2,K), \mathfrak{gl}(2,K)]]=\text{span}\{E_{11}+E_{22}\}.$$
It follows that~$[\mathfrak{gl}(2,K), \mathfrak{gl}(2,K)]$ and~$[\mathfrak{gl}(2,K[x,x^{-1}]), \mathfrak{gl}(2,K[x,x^{-1}])]$ are nilpotent. In particular, $\mathfrak{gl}(2, \mA)$ is solvable.  On the other hand, since $[E_{11},E_{12}]=E_{12}$, we obtain that $\mathfrak{gl}(2,K)$ is  not nilpotent.
For $n=3$, define
$$\mathcal{L}=\text{span}\{E_{ii} + E_{jj}, E_{ij}\mid i\neq j, 1\leq i, j\leq 3\}\subseteq \mathfrak{gl}(3,K).$$
We then have $$[\mathfrak{gl}(3,K), \mathfrak{gl}(3,K)]\supseteq [\mathcal{L}, \mathcal{L}]=\mathcal{L},$$ so $\mathfrak{gl}(3,K)$ is not solvable. This implies that for all $n\geq 3$, the Lie algebra $\mathfrak{gl}(n, \mA)$ is not solvable, thus finishing the proof.
\end{proof}

The remainder of this section is to characterize Lie solvable and Lie nilpotent Leavitt path algebras. To do so, we need to recall the notions of graphs and Leavitt path algebras, and establish useful facts. We refer the reader to \cite{a:lpatfd} and \cite{AAS} for a detailed history and overview of Leavitt path algebras.

A (directed) graph $E = (E^0, E^1, s, r)$ consists of two disjoint sets $E^0$ and $E^1$, called \emph{vertices} and \emph{edges} respectively, together with two maps $s, r: E^1 \longrightarrow E^0$.  The vertices $s(e)$ and $r(e)$ are referred to as the \emph{source} and the \emph{range}
of the edge~$e$, respectively. A vertex~$v$ for which $s^{-1}(v)$ is empty is called a \emph{sink}; a vertex~$v$ is \emph{regular} if $0< |s^{-1}(v)| < \infty$; a vertex~$v$ is an \emph{infinite emitter} if $|s^{-1}(v)| = \infty$; and a vertex~$v$ is \emph{isolated} if $|s^{-1}(v)| =0 = |r^{-1}(v)|$. A graph $E$ is called {\it row-finite} if  $|s^{-1}(v)| < \infty$ for all $v\in E^0$.

A \emph{path of length} $n$ in a graph $E$ is a sequence $p = e_{1} \cdots e_{n}$  of edges $e_{1}, \dots, e_{n}$ such that $r(e_{i}) = s(e_{i+1})$ for $i = 1, \dots, n-1$.  In this case, we say that the path~$p$  starts at the vertex $s(p) := s(e_{1})$ and  ends at the vertex $r(p) := r(e_{n})$, we write $|p| = n$ for the length of $p$.
We denote by $p^0$ the set of its vertices, that is, $p^0:= \{s(e_i)\mid 1\le i\le n\} \cup \{r(e_n)\}$. We consider the elements of $E^0$ to be paths of length $0$.
A path  $p= e_{1} \cdots e_{n}$ of positive length is a \textit{cycle based at} the vertex $v$ if $s(p) = r(p) =v$ and the vertices $s(e_1), s(e_2), \hdots, s(e_n)$ are distinct. A cycle c is called a {\it loop} if $|c| = 1$. Two cycles $c$ and $d$ is called {\it distinct} if $c^0 \neq d^0$. An edge $f$ is an \emph{exit} for a path $p = e_1 \cdots e_n$ if $s(f) = s(e_i)$ but $f \ne e_i$ for some $1 \le i \le n$. A graph $E$ is said to be a {\it no-exit graph} if no cycle in $E$ has an exit.

\begin{defn}\label{LPAs}
For an arbitrary graph $E = (E^0,E^1,s,r)$
and an arbitrary field $K$, the \emph{Leavitt path algebra} $L_{K}(E)$ {\it of the graph}~$E$
\emph{with coefficients in}~$K$ is the $K$-algebra generated
by the union of the set $E^0$ and two disjoint copies of $E^1$, say $E^1$ and $\{e^*\mid e\in E^1\}$, satisfying the following relations for all $v, w\in E^0$ and $e, f\in E^1$:
\begin{itemize}
\item[(1)] $v w = \delta_{v, w} w$;
\item[(2)] $s(e) e = e = e r(e)$ and $r(e) e^* = e^* = e^*s(e)$;
\item[(3)] $e^* f = \delta_{e, f} r(e)$;
\item[(4)] $v= \sum_{e\in s^{-1}(v)}ee^*$ for every regular vertex $v$;
\end{itemize}
where $\delta$ is the Kronecker delta.
\end{defn}
For every path $p = e_1e_2\cdots e_n$, the element $e^*_n\cdots e^*_2e^*_1$ of $L_K(E)$ is denoted by $p^*$. In particular, for every~$v\in E^0$, by~$v^*$ we mean~$v$.
It can be shown (\cite[Lemma 1.7]{ap:tlpaoag05}) that $L_K(E)$ is  spanned as a $K$-vector space by $\{pq^* \mid p, q \mbox{ are paths in $E$ with }$ $ r(p) = r(q)\}$.

The following lemma provides us with a necessary condition for Leavitt path algebras  to be Lie solvable.

\begin{lem}\label{cyc-exi}
Let $K$ be an arbitrary field and $E$ an arbitrary graph such that $L_{K}(E)$ is Lie solvable. Then $E$ is a no-exit graph.
\end{lem}
\begin{proof} Assume that $E$ has a cycle $c$ based at $v$ with an exit $f$. Without loss of generality, we may assume that $s(f) = v$.
By \cite[Theorem 1]{aajz:lpaofgkd}, the set
$$\{c^nf f^*(c^m)^*\mid 1\leq n\leq 3, 1\leq m\leq 3\}$$
is linearly independent in $L_{K}(E)$.  We also note that $c^*f=f^*c=0$, $c^*c = v$, $f^*v = f^*$ and $vf = f$. From these observations, it is straightforward to show  that  the set  $\text{span}\{c^nf f^*(c^m)^*\mid 1\leq n\leq 3, 1\leq m\leq 3\}$ is a $K$-subalgebra of $L_K(E)$ which is isomorphic to the matrix algebra $M_3(K)$ via the map: $c^nf f^*(c^m)^*\mapsto E_{nm}$.  Then, by Lemma~\ref{mat-sol}, we have that $L_{K}(E)$ contains a subalgebra which is not Lie solvable, and so $L_{K}(E)$ is not Lie solvable,  a contradiction.  Therefore $E$ is a no-exit graph, thus finishing the proof.
\end{proof}

We are now in position to provide the main result of this section  characterizing Lie solvable Leavitt path algebras.

\begin{thm}\label{sol-LPAs}
Let $K$ be an arbitrary field and $E$ an arbitrary graph. Then the following holds:

$(1)$ If $\chart(K)=2$, then $L_K(E)$ is Lie solvable if and only if if $E$	is a no-exit graph satisfying the following condition: every vertex $v$ in $E$ is either a sink, or in a cycle whose length is at most $2$, or for each $e\in s^{-1}(v)$, $r(e)$ is either a sink with $r^{-1}(r(e)) = \{e\}$ or in a loop $f$ with $r^{-1}(r(e)) = \{e, f\}$.

$(2)$ If $\chart(K)\neq 2$, then $L_K(E)$ is Lie solvable if and only if $E$ is  a disjoint union of isolated vertices and loops. In this case, $L_K(E) \cong K^{(\Upsilon_1)} \oplus K[x, x^{-1}]^{(\Upsilon_2)}$, where $\Upsilon_1$ is the set of all isolated vertices in $E$ and $\Upsilon_2$ is the set of all loops in $E$.
\end{thm}
\begin{proof}
Assume that~$L_K(E)$ is Lie solvable. By Lemma \ref{cyc-exi}, $E$ is a no-exit graph. Consider the following cases.

{\it Case} 1.  There exist two distinct edges $e$ and $f$ in $E$ such that $s(f) \neq r(f) = s(e)\neq r(e)$ and $s(f)\neq r(e)$. Let $v := r(e)$ and let $A$ be the  $K$-subspace of $L_K(E)$ spanned by $fee^*f^*$, $ee^*$, $ v$, $fee^*$, $ee^*f^*$, $e$, $e^*$,   $fe$,  and $e^*f^*$. Then it is straightforward to show that $A$ is a subalgebra of $L_K(E)$, which is isomorphic to the matrix algebra $M_3(K)$ via the map: $fee^*f^*\longmapsto E_{11}$, $ee^* \longmapsto E_{22}$, $v \longmapsto E_{33}$, $fee^* \longmapsto E_{12}$, $ee^*f^* \longmapsto E_{21}$, $e \longmapsto E_{23}$, $e^* \longmapsto E_{32}$, $fe \longmapsto E_{13}$ and $e^*f^* \longmapsto E_{31}$. By Lemma \ref{mat-sol}, $L_K(E)$ is not Lie solvable, a contradiction.

{\it Case} 2.  There exist two distinct edges $e$ and $f$ in $E$ such that $s(e)\neq r(e)$, $s(f)\neq r(f)$ and $r(e) = r(f) = v \in E^0.$ Let $B$ be the  $K$-subspace of $L_K(E)$ spanned by    $ee^*$, $ff^* $, $v$,  $ef^* $, $fe^* $,  $e $, $e^*$, $f $ and $f^* $.  Then it is straightforward to show that $B$ is a subalgebra of $L_K(E)$, which is isomorphic to $M_3(K)$ via the map: $ee^*\longmapsto E_{11}$, $ff^* \longmapsto E_{22}$, $v\longmapsto E_{33}$,  $ef^* \longmapsto E_{12}$, $fe^* \longmapsto E_{21}$,  $e \longmapsto E_{13}$, $e^* \longmapsto E_{31}$, $f \longmapsto E_{23}$ and $f^* \longmapsto E_{32}$. By Lemma \ref{mat-sol}, $L_K(E)$ is not Lie solvable, a contradiction.

In any case, we arrive at a contradiction, and so $E$ is a no-exit graph satisfying the following condition: every vertex $v$ in $E$ is either a sink, or in a cycle whose length is at most $2$, or for each $e\in s^{-1}(v)$, $r(e)$ is either a sink with $r^{-1}(r(e)) = \{e\}$ or in a loop $f$ with $r^{-1}(r(e)) = \{e, f\}$. In particular, no paths of positive length end at an infinite emitter.

Let~$I$ be the ideal of~$L_K(E)$ generated by~$\Upsilon_1 \cup \{c^0\mid c\in \Upsilon_2\}$, where~$\Upsilon_1$ is the set of all sinks in~$E$ and $\Upsilon_2$ is the set of all distinct cycles in $E$.  By \cite[Corollary 2.7.5 (i)]{AAS}, we have
\begin{equation}\label{i-form}
  I \cong (\bigoplus_{v\in \Upsilon_1}M_{n(v)}(K)) \oplus (\bigoplus_{c\in \Upsilon_2}M_{m(c)}(K[x, x^{-1}])),
\end{equation}
where $n(v)$ is the number of all paths ending in the sink $v$, and $m(c)$ is the number of all paths ending in a fixed vertex of the cycle $c$ which do not contain the cycle itself. Since~$L_{K}(E)$ is Lie solvable, we have~$n(v)\leq 2$ for every $v\in \Upsilon_1$ and~$m(c)\leq 2$ for every $c\in \Upsilon_2$.

Let $e$ be an edge in $E$. Since $r(e)$ is either a sink, or in a cycle whose length is at most 2, we must have $r(e)\in I$,  $e = er(e)\in I$ and $e^* = r(e)e^*\in I$. For every regular vertex $v\in E^0$, we have $v = \sum_{e \in s^{-1}(v)} ee^* \in I$.  Combining these observations and the fact that no paths of positive length end at an infinite emitter, we deduce
\begin{equation}\label{for-row-fi}
L_K(E)/I\cong K^{(S)}
 \end{equation}
 as $K$-algebras, where $S$ is the set of all infinite emitters in $E$. Now we have to consider the characteristic of the field~$K$.

(1) If~$\chart(K)=2$, then we already obtain the necessary condition. Next we prove the sufficient condition. Since $K^{(S)}$ is a commutative algebra, we immediately obtain that $[L_K(E), L_K(E)] \subseteq I$.   Then by Lemma \ref{mat-sol} (2) and by the fact that a direct sum of an arbitrary family of Lie solvable algebras is also Lie solvable, we deduce that $I$ is Lie solvable, and so $[L_K(E), L_K(E)]$ is Lie solvable.  On the other hand, it is well known that $L_K(E)$ is Lie solvable if and only if so is $[L_K(E), L_K(E)]$. From these observations we have that $L_K(E)$ is Lie solvable, showing the sufficient condition.

(2) If $\chart(K)\neq 2$, then since $L_K(E)$ is Lie solvable, we immediately obtain that $I$ is Lie solvable. Then, by Lemma~\ref{mat-sol} (1), we have~$n(v)=1$ for every $v\in \Upsilon_1$, and $m(c) = 1$ for  every $c\in \Upsilon_2$. In other words, every vertex $v\in \Upsilon_1$ is isolated and every cycle $c\in \Upsilon_2$ is a loop such that~$s^{-1}(s(c))=r^{-1}(s(c))=\{c\}$. This implies that $E$ is exactly  a disjoint union of isolated vertices and loops. In this case, we have
$L_K(E) \cong K^{(\Upsilon_1)} \oplus K[x, x^{-1}]^{(\Upsilon_2)}$. The converse immediately follows from the note that $L_K(E)$ is commutative, thus finishing the proof.
\end{proof}

By the germane ideas in the proof of Theorem \ref{sol-LPAs},  we  determine the structure of the Leavitt path algebra $L_K(E)$ of a row-finite graph $E$ over a field $K$ of characteristic $2$ which is Lie solvable.


\begin{cor}\label{sol-LPAs1}
Let $K$ be a field with $\chart(K)= 2$ and let $E$ be a row-finite graph. Then $L_K(E)$ is Lie solvable if and only if $$L_K(E) \cong (\bigoplus_{v\in \Upsilon_1}M_{n(v)}(K)) \oplus (\bigoplus_{c\in \Upsilon_2}M_{m(c)}(K[x, x^{-1}])),$$ where $\Upsilon_1$ is the set of all sinks in $E$, $\Upsilon_2$ is the set of all distinct cycles in $E$, $n(v)$ is both at most $2$ and the number of all paths ending in the sink $v$, and $m(c)$ is both at most $2$ and the number of all paths ending in a fixed vertex of the cycle $c$ which do not contain the cycle itself.	
\end{cor}
\begin{proof}

Assume that $L_K(E)$ is Lie solvable. By repeating the method described in the proof of Theorem \ref{sol-LPAs}, we have that $E$ is a no-exit graph satisfying the following condition: every vertex $v$ in $E$ is either a sink, or in a cycle whose length is at most $2$, or for each $e\in s^{-1}(v)$, $r(e)$ is either a sink with $r^{-1}(r(e)) = \{e\}$ or in a loop $f$ with $r^{-1}(r(e)) = \{e, f\}$.  Let~$I$ be  the ideal of~$L_K(E)$ generated by~$\Upsilon_1 \cup \{c^0\mid c\in \Upsilon_2\}$, where~$\Upsilon_1$ is the set of all sinks in~$E$ and $\Upsilon_2$ is the set of all distinct cycles in $E$. Then we have
$$I \cong (\bigoplus_{v\in \Upsilon_1}M_{n(v)}(K)) \oplus (\bigoplus_{c\in \Upsilon_2}M_{m(c)}(K[x, x^{-1}])),$$
where $n(v)$ is both at most $2$ and the number of all paths ending in the sink $v$, and $m(c)$ is both at most $2$ and the number of all paths ending in a fixed vertex of the cycle $c$ which do not contain the cycle itself. As was shown in the proof of Theorem \ref{sol-LPAs} (1), $v\in I$ for all vertex $v\in E^0$ with $|s^{-1}(v)| < \infty$, and $e,\, e^*\in I$ for all $e\in E$. Then, since $E$ is a row-finite graph, we have $E^0 \cup \{e, \, e^*\mid e\in E^1\}\subseteq I$, and so $L_K(E) = I$, proving the necessary condition. The sufficient condition follows from Lemma \ref{mat-sol} and  the well known fact that a directed sum of an arbitrary family of Lie solvable algebras is also Lie solvable, thus finishing the proof.
\end{proof}

For clarification, we illustrate Theorem \ref{sol-LPAs} and Corollary \ref{sol-LPAs1} by presenting the following examples.

\begin{exas}
(1) Let $K$ be a field and $C = (C^0, C^1, s, r)$ the graph with $C^0 = \{v, w_n\mid n\in \mathbb{N}\}$, $C^1 = \{e_n\mid n \in \mathbb{N}\}$, $s(e_n) = v$ and $r(e_n) = w_n$ for all $n$. (So $C$ is the ``infinite clock" graph described in \cite[Example 1.6.12]{AAS}.) Then, by Theorem \ref{sol-LPAs}, $L_K(C)$ is Lie solvable if and only if $\chart(K) = 2$.

(2) Let $K$ be a field and  $E = (E^0, E^1, s, r)$ the graph with $E^0 = \{u, v_1, v_2,  v_3, v_4\}$, $E^1 = \{e_1, e_2, e_3, e_4, f, g\}$, $s(e_i) = u$, $r(e_i) = v_i$, $s(f) = r(f) = v_3$, and  $s(g) = r(g) =v_4$. We then have $\Upsilon_1 = \{v_1, v_2\}$, $\Upsilon_2 = \{f, g\}$ and $$L_K(E)\cong M_2(K) \oplus M_2(K) \oplus M_2(K[x, x^{-1}])\oplus M_2(K[x, x^{-1}]).$$ By Theorem \ref{sol-LPAs} and Corollary \ref{sol-LPAs1}, $L_K(E)$ is Lie solvable if and only if $\chart(K) = 2$.
\end{exas}


As a corollary of Theorem \ref{sol-LPAs}, we obtain a complete description of Lie nilpotent Leavitt path algebras.

\begin{cor}\label{nil-LPAs}
For every field $K$ and every graph $E$, the following conditions are equivalent:

$(1)$ $L_K(E)$ is Lie nilpotent;

$(2)$ $E$ is  a disjoint union of isolated vertices and loops;

$(3)$ $L_K(E) \cong K^{(\Upsilon_1)} \oplus K[x, x^{-1}]^{(\Upsilon_2)}$, where $\Upsilon_1$ is the set of all isolated vertices in $E$ and $\Upsilon_2$ is the set of all loops in $E$.
\end{cor}
\begin{proof} (1)$\Longrightarrow$(2). Assume that $L_K(E)$ is Lie nilpotent. Since every nilpotent Lie algebra is solvable, $L_K(E)$ is Lie solvable.  Let $I$ be the ideal of~$L_K(E)$ generated by~$\Upsilon_1 \cup \{c^0\mid c\in \Upsilon_2\}$, where~$\Upsilon_1$ is the set of all sinks in~$E$ and $\Upsilon_2$ is the set of all distinct cycles in $E$. Then  we have  $$I \cong (\bigoplus_{v\in \Upsilon_1}M_{n(v)}(K)) \oplus (\bigoplus_{c\in \Upsilon_2}M_{m(c)}(K[x, x^{-1}])),$$  where $n(v)$ is the number of all paths ending in the sink $v$, and $m(c)$ is the number of all paths ending in a fixed vertex of the cycle $c$ which do not contain the cycle itself. Since $L_K(E)$ is Lie nilpotent, $I$ is also Lie nilpotent. Then, by Lemma \ref{mat-sol},  we have~$n(v)=1$ for every $v\in \Upsilon_1$, and $m(c) = 1$ for  every $c\in \Upsilon_2$. In other words, every vertex $v\in \Upsilon_1$ is isolated and every cycle $c\in \Upsilon_2$ is a loop such that~$s^{-1}(s(c))=r^{-1}(s(c))=\{c\}$. This implies that $E$ is exactly  a disjoint union of isolated vertices and loops.
	
(2)$\Longrightarrow$(3). It is obvious.

(3)$\Longrightarrow$(1). It follows from the note that $L_K(E)$ is commutative, thus finishing the proof.	
\end{proof}

It is well known that the Lie solvability of  a Lie algebra~$\mathcal{L}$ does not imply the Lie nilpotency of~$[\mathcal{L},\mathcal{L}]$ in general.  As another corollary of Theorem \ref{sol-LPAs}, we conclude the article with the following interesting result, showing the above fact is true for the associated Lie algebra of a Leavitt path algebra.


\begin{cor}\label{nil-LPAs1}
For a field $K$ and a graph $E$, $L_K(E)$ is Lie solvable if and only if $[L_{K}(E),L_{K}(E)]$ is Lie nilpotent.
\end{cor}
\begin{proof}
Assume that~$L_{K}(E)$ is Lie solvable. Consider the following two cases.

{\it Case} 1. $\chart(K)\neq 2$. By Theorem~\ref{sol-LPAs}, $L_{K}(E)$ is commutative, and so $[L_{K}(E),L_{K}(E)]= 0$, in particular,  $[L_K(E), L_K(E)]$ is Lie nilpotent.

{\it Case} 2. $\chart(K) = 2$. By repeating the method described in the proof of Theorem \ref{sol-LPAs}, we note that $E$ is a no-exit graph satisfying the following condition: every vertex $v$ in $E$ is either a sink, or in a cycle whose length is at most $2$, or for each $e\in s^{-1}(v)$, $r(e)$ is either a sink with $r^{-1}(r(e)) = \{e\}$ or in a loop $f$ with $r^{-1}(r(e)) = \{e, f\}$. In particular, no paths of positive length end at an infinite emitter.
 Let $I$ be  the ideal of~$L_K(E)$ generated by~$\Upsilon_1 \cup \{c^0\mid c\in \Upsilon_2\}$, where~$\Upsilon_1$ is the set of all sinks in~$E$ and $\Upsilon_2$ is the set of all distinct cycles in $E$. Then we have
\begin{equation}\label{coro-i-f}
  I \cong (\bigoplus_{v\in \Upsilon_1}M_{n(v)}(K)) \oplus (\bigoplus_{c\in \Upsilon_2}M_{m(c)}(K[x, x^{-1}])),
\end{equation} where  $n(v)$ is both at most $2$ and the number of all paths ending in the sink $v$, and $m(c)$ is both at most $2$ and the number of all paths ending in a fixed vertex of the cycle $c$ which do not contain the cycle itself. As was shown in the proof of Theorem \ref{sol-LPAs} (1), $v\in I$ for every vertex $v\in E^0$ with $|s^{-1}(v)| < \infty$, and $e,\, e^*\in I$ for every $e\in E$. Consequently, $p$ and $p^*\in I$ for every path $p$ of positive length in $E$.

We next claim that $[L_K(E), L_K(E)] = [I, I]$. It  suffices  to prove that $$[L_K(E), L_K(E)] \subseteq [I, I].$$
Let $\alpha$ be an arbitrary element in $L_K(E)$.  Write
$\alpha = \sum^m_{j= 1}k_jv_j+
\sum^n_{i=1}l_i p_iq^*_i$, where $k_j,\, l_i\in K$, $v_j$'s are distinct vertices in $E^0$,  and $p_i$ and $q_i$ are paths in $E$ such that $r(p_i) = r(q_i)$, and for each $1\le i\le n$, either $|p_i| \ge 1$ or $|q_i| \ge 1$. By the above note,  $ p_iq^*_i \in I$ for all $i$. Let $v\in E^0$. We then have $[\sum^m_{j= 1}k_jv_j, v] = 0$ and
$$[\alpha, v] = [\sum^m_{j= 1}k_jv_j, v] + [\sum^n_{i=1}l_i p_iq^*_i, v] =  [\sum^n_{i=1}l_i p_iq^*_i, v]=\sum^n_{i=1}l_i (p_iq^*_iv - vp_iq^*_i).$$
If $v$ is either regular  or is a sink, then we have  $v\in I$, and so
$$[\alpha, v] =  [\sum^n_{i=1}l_i p_iq^*_i, v]\in [I, I]. $$
Now we assume that $v$ is an infinite emitter. If~$s(p_i)=s(q_i)=v$, or if~$v$ is neither~$s(p_i)$ nor~$s(q_i)$,  then we have
$$p_iq^*_iv - vp_iq^*_i=0\in [I, I]. $$

We first consider the case when~$s(q_i)=v$ and~$s(p_i)\neq v$. If $q_i = v$, then we have $r(p_i) = r(q_i) = v$ and $|p_i|\ge 1$, that means,~$p_i$ is a path of positive length that ends at an infinite emitter, a contradiction. This implies that $|q_i| \ge 1$. Therefore, we have either $q_i = e_i$ with $s(e_i) = v$ and $r(e_i)$ is a sink such that $r^{-1}(r(e_i)) = \{e_i\}$, or  $q_i = e_if^{n_i}$, where $n_i\ge 0$, $e_i\in E^1$ with $s(e_i) = v$,  $f$ is a loop, and $r^{-1}(r(e_i)) = \{e_i, f\}$, where by~$f^0$ we mean~$r(e_i)$.  For the first subcase, since~$s(p_i)\neq v $ and~$r(p_i)= r(q_i)$,  we deduce~$p_i=r(q_i)\neq v$ and
$$p_iq^*_iv - vp_iq^*_i=q^*_i=e^*_i
=r(e_i)e^*_i-e^*_ir(e_i)=[r(e_i),e^*_i]\in [I, I].$$
For the second subcase,  since~$s(p_i)\neq v$ and~$r_{p_i}=r(q_i)=r(e_i)$, we deduce~$p_i=f^{m_i}$ for some integer~$m_i\geq 0$. Denote $(f^*)^n$ by $f^{-n}$ for every positive integer~$n$. Then since~$r(e_i)$ lies in~$I$, we know~$f^{m}e^*_i=r(e_i)f^{m}e^*_i\in I$ for every integer~$m$. Thus we obtain
$$p_iq^*_iv - vp_iq^*_i=f^{m_i-n_i}e^*_i=[r(e_i),f^{m_i-n_i}e^*_i]\in [I, I].$$

Then we consider the case when~$s(p_i)=v$ and $s(q_i)\neq v$. The proof of this case is almost the same to the above case. And just for the reader's convenience, we still offer a detailed proof. If~$p_i=v$, then we have $ r(q_i)=r(p_i) = v$ and $|q_i|\ge 1$, which means that~$q_i$ is a path of positive length that ends at an infinite emitter, a contradiction. This implies that $|p_i| \geq 1$. If~$p_i=e_i\in E^1$ such that~$r(e_i)$ is a sink, then we deduce~$q_i=r(e_i)$ and
$$p_iq^*_iv - vp_iq^*_i=-p_i=-e_i
=r(e_i)e_i-e_ir(e_i)=[r(e_i),e_i]\in [I, I].$$
If~$p_i = e_if^{n_i}$, where $n_i\ge 0$, $e_i,f\in E^1$ and $f$ is a loop, then~$q_i=f^{m_i}$ for some integer~$m_i\geq 0$, and so we obtain
$$p_iq^*_iv - vp_iq^*_i=-e_i f^{n_i-m_i}=[r(e_i),e_i f^{n_i-m_i}]\in [I, I].$$
Therefore, for every~$v\in E^0$, we have~$[\alpha, v] =  [\sum^n_{i=1}l_i p_iq^*_i, v]\in [I, I]$,
and
$$[L_{K}(E),v]\subseteq [I,I].$$
Since~$L_{K}(E)=\sum_{v\in E^0}Kv+I$, we deduce
\begin{align*}
  [L_{K}(E),L_{K}(E)]&=[\sum_{v\in E^0}Kv+I,\sum_{v\in E^0}Kv+I]&\\
&=[\sum_{v\in E^0}Kv,\sum_{v\in E^0}Kv+I]+[I,I]\subseteq [I,I],&
\end{align*}
 thus finishing the proof of the claim.

As was shown in  the proof of Lemma~\ref{mat-sol} (2) that  $[\mathfrak{gl}(2,K), \mathfrak{gl}(2,K)]$ and $[\mathfrak{gl}(2,K[x,x^{-1}]), \mathfrak{gl}(2,K[x,x^{-1}])]$ are nilpotent, and so by Equation~\eqref{coro-i-f}, $[I, I]$ is Lie nilpotent. This implies that $[L_{K}(E),L_{K}(E)]$ is also Lie nilpotent, thus finishing the proof.
\end{proof}

\section{Acknowledgement} The authors thank Yuqun Chen for valuable suggestions on the exposition of the paper.

\newcommand{\noopsort}[1]{}

\end{document}